\documentclass{article}

\usepackage[utf8]{inputenc}
\usepackage{bbm}
\usepackage{amsmath}               % great math stuff
\usepackage{amsfonts,amssymb}              % for blackboard bold, etc
\usepackage{amsthm}               

\newtheorem{theorem}{Theorem}
\newtheorem{lemma}{Lemma}

\newtheorem{definition}{Definition}
\newtheorem{example}{Example}

\newcommand{\R}{\mathbb R} 
\renewcommand{\Re}{\mathrm{Re}}
\newcommand{\C}{\mathbb C}

\newcommand{\E}{\mathbf {E}}

\newcommand{\N}{\mathbb {N}}
\newcommand{\ind}{\mathbb{I}}
\newcommand{\Cov}{\mathbf{Cov}}

\newcommand{\vf}{\varphi}

\title{A short memory condition for  infinitely divisible random fields}
\author{Vitalii Makogin\thanks{Dr. Vitaly Makogin unexpectedly passed away on the 8th of May 2024.}, Evgeny Spodarev\thanks{Institute of Stochastics,
Ulm University, Germany; evgeny.spodarev@uni-ulm.de}}
\date{\today}

\begin{document}

\maketitle

\begin{abstract}
    This small note yields a sufficient condition for the short range dependence of measurable stationary infinitely divisible moving average  random fields with $d$--dimensional index space. Here, the short/long range dependence concept in borrowed from the paper \cite{Kulik_Spo_21}. In the special case of symmetric stable moving averages, our new condition coincides with the one from paper \cite{MaOeRappSpo21}. 
\end{abstract}
%-----------------------------------------------------------------------------------------------------------------------------------------------------------------------------------------------------------------------

 \\
 
%    General info
%-----------------------------------------------------------------------------------------------------------------------------------------------------------------------------------------------------------------------

\section{Memory of random fields and their level sets}\label{sect:Intro}

A random field $X = \{X(t), \; t\in \R^d \}$  is called { \it infinitely divisible} (ID, for short) if all its finite dimensional distributions are
infinitely divisible, see e.g.  \cite{SteutelHarn04,Sato13} for an introduction into infinite divisibility. $X$ is said to be { \it  stationary} if all its finite dimensional distributions are translation invariant (with respect to space variable $t$). In this note, we give a sufficient condition for stationary ID $X$ to have a short memory.   Many definitions of the memory of random functions are available in research literature, confer e.g. \cite{MR3075595,Samorodnitsky_LRD16}.   However,  the only short/long range dependence concept which is invariant with respect to monotone transformations on the marginals of the  field $X$  has been recently introduced in \cite{Kulik_Spo_21}.  It depends on the level sets $\{ t: X(t)>u \}$  of $X$  in the following way:

\begin{definition}\label{def:SRD}
A measurable stationary random field $X$ is called short range dependent (SRD) if
\begin{equation*}
  \int_{\R^d}  \int_\R\int_\R  \left|  \Cov(\mathbbm{1}\{X(t)>u\},\mathbbm{1}\{X(0)>v\})  \right|\,   \mu(du)\mu(dv) \, dt < \infty
\end{equation*}
for all probability measures $\mu$ on $\R$. If the above integral is infinite for some choice of $\mu$ the field $X$ is named long range dependent (LRD).  
\end{definition}
In addition, the above definition can be applied to heavy--tailed fields $X$ since it does not require the integrability of $X(0)$. 
The  target statistic from Definition \ref{def:SRD} is the volume of level sets of $X$. In multivariate extreme value theory, it appears naturally e.g. in the study of  the asymptotic behaviour of estimates of the tail dependence coefficient, cf. the recent paper \cite{OestRapp23}. A statistical procedure for the detection of this kind of SRD/LRD is discussed in \cite{OeRappSpo23}.

In paper \cite{MaOeRappSpo21}, implications of Definition \ref{def:SRD} for $\alpha$-- and max--stable processes are shown. 
There, necessary and/or sufficient conditions of SRD/LRD are given in terms of the kernel function $f$ in the spectral decomposition of $X$.

Here, we prove a  similar sufficient condition for stationary ID  random fields $X$ which possess a moving average  integral representation.
The paper is organized as follows: preliminaries on ID fields are given in the next section. The main SRD condition is formulated in Theorem \ref{thm:main} of Section \ref{sect:Main}.  An example of symmetric stable moving averages is considered there as well. Finally, the proof of Theorem \ref{thm:main} (subdivided into a number of auxiliary  results) is given in Section \ref{sect:Proofs}.

%-----------------------------------------------------------------------------------------------------------------------------------------------------------------------------------------------------------------------

\section{Infinitely divisible random fields}\label{sect:ID}

Let $\Lambda$ be a stationary ID independently scattered random measure with Lebesgue control measure and L\'{e}vy characteristics $(a_0, b_0, \nu_0),$ where $a_0 \in \R$, $ b_0\geq 0 $ and $\nu_0$ is a L\'{e}vy measure.   %\cite{Ros18,BPTh20}
In the sequel, we follow the exposition of ID random functions in \cite{RajRos89}:
\begin{definition}
A function $f:\R^d\to \R$ is called $\Lambda$-integrable if 
\[\int_{\R^d}|f(x)|\left|a_0 +\int_{\R}\left(    \mathbbm{1}_{[-1,1]}(yf(x))-  \mathbbm{1}_{[-1,1]}(y) \right)  y  \nu_0(dy)\right|dx<\infty,\]
\[b_0^2\int_{\R^d} f^2(x)dx<\infty,\]
\[\int_{\R^d}\int_{\R}\min\left(1,y^2 f^2(x)\right)\nu_0(dy)dx<\infty.\]
\end{definition}
For $b_0>0$, $\Lambda $ possesses a non--trivial Gaussian part (in addition to the jump part controlled by the L\'evy measure $\nu_0$).

Let $X = \{X(t), t\in \R^d \}$ be a real--valued measurable stationary ID moving average random field   with spectral representation
\begin{equation}\label{eq:fieldX}
X(t) =\int_{\R^d} f(t-x)\Lambda(dx),\quad t \in  \R^d,
\end{equation}
where $f$ is $\Lambda$-integrable.

The characteristic function $\vf_{X(0)}$ of $X(0)$ is given by
\begin{equation*}
\label{vf:def}
    \vf_{X(0)}(u)=\exp\left(-\int_{\R^d}K(u f(s))ds\right),
\end{equation*}
where
\begin{equation*}
\label{K:def*}
K(s)=-isa_0+ \frac{1}{2} s^2b_0 - \int_{\R}\left(e^{isy}-1-isy \ind_{[-1,1]}(y)\right)\nu_0(dy), \quad s\in\R.
\end{equation*}
Functions  $K$ and $\Re \{K \} $ are known to be continuous negative definite functions. Obviously,  $\Re \{ K(s) \} =\Re \{ K (|s|) \}$, $s\in \R$ and $\Re \{ K (s) \}=\psi(s^2),s\geq 0,$ where $\psi$ is a {\it Bernstein function}, see \cite{Schilling12} for its definition and properties.

%Let $\nu_0$ has a density $v_0,$ then $X(0)$ is the infinitely random variable with L\'{e}vy characteristics $(a_1, b_1, \nu_1),$
%where
%$$a_1=\int_{\R^d}U(f(-x))dx,$$
%$$b_1=\int_{\R^d}|f(-x)|^2 dx,$$
%$\nu_1$ has a density $$v_1(x)=\int_{supp(f)}\frac{1}{|f(x)|}\lambda_0\left(\frac{x}{f(x)}\right)dx.$$

Due to the linearity of ID stochastic integrals,  random variable
$$s_1X(t)+s_2X(0)=\int_{\R^d}\left(s_1f(t-x)+s_2f(-x)\right)\Lambda (dx)$$ 

is  ID   with characteristic function
\begin{equation*}
\label{vft:def}
\vf_t(s_1,s_2)=\exp\left(-\int_{\R^d}K(s_1f(t-x)+s_2f(-x)) \, dx\right),\quad (s_1,s_2)\in \R^2.
\end{equation*}

%-----------------------------------------------------------------------------------------------------------------------------------------------------------------------------------------------------------------------

\section{Main result}\label{sect:Main}

Let $| \cdot |$ be the Lebesgue measure in $\R^d$. Denote by $L^{p}(\R^d)$ the space of all  functions 
$f:\R^d\to \R$ such that $|f|^p$ is Lebesgue--integrable, $p>0$. Let $a\wedge b$ denote the minimum of $a$ and $b$.

Introduce \[\sigma_f^2(s):=\int_{\R^d}\Re\{K (s f(-x))\}\, dx, \quad s\in \R,\]
\[\rho_t(s_1,s_2):=\frac{1}{\sigma_f(s_1)\sigma_f(s_2)}\int_{\R^d}\sqrt{\Re\{ K(s_1 f(t-x))\}\Re\{ K (s_2 f(-x))\}}  \, dx \]
for $s_1,s_2\in \R$.
Note that $0\leq \rho_t(s_1,s_2)\leq 1$ for all $s_1,s_2\in \R$, $t\in \R^d.$
%Moreover, $\sigma_f^2(s)=0.$ 
In addition,  set \[\tilde{\rho}_t:=\sup_{(s_1,s_2)\in \R^2}\rho_t(s_1,s_2)\] and $A_{\tilde{\rho}}=\{t\in \R^d:  \tilde{\rho}_t \leq \tilde{\rho}  \}$ for a fixed $\tilde{\rho}\in (0,1)$.
Let $A_{\tilde{\rho}}^c=\{t\in \R^d:  \tilde{\rho}_t > \tilde{\rho}  \}$ be the compliment of  $A_{\tilde{\rho}}$ in $\R^d$.

\begin{theorem}\label{thm:main}
Let $X$ be  a measurable stationary ID moving average random field  \eqref{eq:fieldX}.
Assume that there exists $\tilde{\rho}<1$ such that $| A^c_{\tilde{\rho}}|<\infty$ and 
\begin{equation}
\label{thm1:eq1}    \int_{\R_+}\frac{\sigma_f(s)}{s} e^{-(1-\tilde{\rho})\sigma_f^2(s)}ds <\infty.
\end{equation}
Then, $X$ is SRD if 
$\int_{\R^d}\tilde\rho_t dt<\infty.$
\end{theorem}
The proof of this result including several auxiliary lemmata will be given in the next section. Notice that $\sigma_f^2(s)=-\log|\vf_{X(0)}(s)|.$
Then the integral within \eqref{thm1:eq1} rewrites as
$$\int_{\R_+}\frac{|\log|\vf_{X(0)}(s)||^{1/2}}{s} |\vf_{X(0)}(s)|^{1-\tilde{\rho}}ds.$$

Now let us explain how to check the above condition \eqref{thm1:eq1} in the special case of an   $\alpha$--stable symmetric integrator measure $\Lambda$,  cf. the book \cite{Samorodnitsky}:

\begin{example}[Stable laws]
Let $X=\{X(t),t\in \R^d\}$ be a measurable stationary symmetric  $\alpha-$stable moving average random field \eqref{eq:fieldX} with $\alpha\in (0,2)$ and a  continuous bounded kernel $f\in L^{\alpha}(\R^d)$.
Then  \[\sigma_f^2(s)=|s|^{\alpha}  \| f \|_\alpha^\alpha, \quad s\in \R, \]
where $\| f \|_\alpha:=  \left( \int_{\R^d} |f(x) |^\alpha \, dx \right)^{1/\alpha}$,
\begin{align*}
\rho_t(s_1,s_2)&=\frac{1}{ |s_1|^{\alpha/2}  |s_2|^{\alpha/2}    \| f \|_\alpha^\alpha }\int_{\R^d}\sqrt{   |s_1|^{\alpha}  |s_2|^{\alpha}    | f(t-x) f(-x) |^\alpha      }dx\\
&=\frac{1}{ \| f \|_\alpha^\alpha }\int_{\R^d} | f(t-x) f(-x) |^{\alpha/2} \,     dx=\tilde \rho_t, \quad t\in \R^d.
\end{align*}

Moreover, if $f$ is a continuous bounded  function then $\tilde \rho_t $ is a continuous function at the origin as well (by Lebesgue dominated convergence   theorem). Hence,  there exists $\tilde{\rho}<1$ such that $| A_{\tilde{\rho}}^c|<\infty$ (the set $A_{\tilde{\rho}}^c$ is a neighbourhood of the origin).
Furthermore,
\begin{align*}
 \int_{\R_+}\frac{\sigma_f(s)}{s} e^{-(1-\tilde{\rho})\sigma_f^2(s)}ds&=  \| f \|_\alpha^{\alpha/2}  \int_{\R_+}s^{\alpha/2-1}  e^{-(1-\tilde{\rho})s^{\alpha} \Sigma_f^2}ds<\infty
 , \quad \alpha\in (0,2).
\end{align*}

Hence, the assumptions of Theorem \ref{thm:main} hold true. The  SRD condition $$\int_{\R^d}\tilde\rho_t dt<\infty$$ 
reads  (by Fubini theorem) as $f\in L^{\alpha/2}(\R^d)$ which coincides with  \cite[Theorem 3.4]{MaOeRappSpo21}.
\end{example}

%-----------------------------------------------------------------------------------------------------------------------------------------------------------------------------------------------------------------------

\section{Proof of Theorem \ref{thm:main}}\label{sect:Proofs}

Notice that by \cite[Theorem 4.12]{Schilling12}
the function $$1-e^{itx}-itx \ind_{[-1,1]}(x),\;  t\in\R$$ is negative definite for any $x\in \R$.  We formulate the following

\begin{lemma}[\cite{Schilling12}]
\label{lmm0}
Let $\psi:\R^d\to \C$ be a continuous negative definite function. Then for any $x,y\in \R^d$ it holds
\begin{itemize}
    \item $\Re\{\psi(x)\}\geq 0,$ $\psi(x)=\overline{\psi(-x)}$, 
   % \item  $|\psi(x)|\leq 2 \sup_{|y|\leq 1}(1+|x|^2),$
    \item $\left|\psi(x)+\psi(y)-\psi(x\pm y)\right|\leq 2\sqrt{\Re\{\psi(x)\}}\sqrt{\Re\{\psi(y)\}}, $
    \item $\left|\Re \{\psi(x)\}+\Re \{\psi(x)\}-\Re \{\psi(x\pm y)\}\right|\leq 2\sqrt{\Re\{\psi(x)\}}\sqrt{\Re\{\psi(y)\}}.$
    %\item $\sqrt{|\psi(x+y)|}\leq \sqrt{|\psi(x)|}+\sqrt{|\psi(y)|}.$
   % \item $\psi(\lambda x)=|\lambda|^\alpha \psi(x), x\in \R^d, \lambda\in \R$ for $\alpha-$stable distributions.
\end{itemize}
\end{lemma}

\noindent  The next inequality is a corollary of the above statements:
\begin{lemma}
\label{lmm2}
For any continuous negative definite function $\psi:\R^d\to \C$  it holds
\begin{equation*}
     \Re\{\psi(x+y)\} \wedge (\Re\{\psi(x)+\psi(y)\}) \geq \left(\sqrt{\Re\{\psi(x)\}}-\sqrt{\Re\{\psi(y)\}}\right)^2.
\end{equation*}    
\end{lemma}
\begin{proof} By Lemma \ref{lmm0}, we get
 \begin{align*}
     &\Re\{\psi(x+y)\} \wedge (\Re\{\psi(x)+\psi(y)\})\\ &=\Re\{\psi(x)+\psi(y)\}-(\Re\{\psi(x)\}+\Re\{\psi(y)\}-\Re\{\psi(x+y)\})_+\\
     &\geq \Re\{\psi(x)+\psi(y)\}- 2\sqrt{\Re\{\psi(x)\}}\sqrt{\Re\{\psi(y)\}}.
 \end{align*}
\end{proof}

%\begin{remark} Note that 
%$|1-e^{iux}+1-e^{ivx}-(1-e^{i(u+v)x})|^2=$
%\begin{align*}
%    &K(u)+K(v)-K(u+v)=-b_0uv+\int_{x\neq 0}\left(1-\cos(iux)+1-\cos(ivx)-(1-\cos(u+v)x)\right)\nu_0(dx)\\
%    &-i\int_{x\neq 0}\left(\sin(ux)+\sin(vx)-\sin((u+v)x)\right)\nu_0(dx)=\Re K(u-v)-\Re K(u+v)+b_0 uv\\
%    &.
%\end{align*}
%\end{remark}

\noindent The following evident bound will be used in the sequel: for any $z_1,z_2 \in \C$ 
\begin{equation}\label{eq:ez}
    \left|e^{-z_1}-e^{-z_2}\right|\leq e^{-(\Re\{z_1\}\wedge \Re\{z_2\} )}|z_1-z_2|.
\end{equation}

\begin{lemma} For all $s_1,s_2\in \R$ it holds 
\label{lmm3}
\begin{align*}
\nonumber    &\left|\E e^{i \left(s_1X(t)+s_2X(0)\right)}-\E e^{i s_1 X(t)} \E e^{i s_2X(0)}\right|\\
%\label{lmm3:eq1} 
&\leq \exp\left[-\int_{\R^d}\left(\sqrt{\Re\{K (s_1 f(t-x))\}}- \sqrt{\Re\{K(s_2 f(-x))\}}\right)^2 dx\right]\\
%\label{lmm3:eq2}    
&\times 2  \int_{\R^d}\sqrt{\Re\{K (s_1 f(t-x))\}}\sqrt{\Re\{K(s_2 f(-x))\}}dx.\\
%\nonumber    &\left|\E e^{i \left(s_1X(t)+s_2X(0)\right)}-\E e^{i s_1 X(t)} \E e^{i s_2X(0)}\right|\\
%\nonumber &\leq \exp\left[-\int_{\R^d}\left(\Re\{K (s_1 f(t-x)+s_2f(-x))\}dx\right)\right]\\
%\label{lmm3:eq3}  
  %  &\times  4\int_{\R^d}\int_{z\neq 0}\left|\sin\left(\frac{z s_1 f(t-x)}{2}\right)\sin\left(\frac{z s_2 f(-x) }{2}\right)\right|\nu_0(dz).
\end{align*}
\end{lemma}
\begin{proof}
Employing inequality \eqref{eq:ez}, we get
\begin{align}
\nonumber    &\left|\E e^{i \left(s_1X(t)+s_2X(0)\right)}-\E e^{i s_1 X(t)} \E e^{i s_2X(0)}\right|\\
\nonumber    &= \left| \exp\left[-\int_{\R^d}K \left(s_1f(t-x)+s_2f(-x)\right)dx\right]\right.\\
\nonumber    &\left.-\exp\left[-\int_{\R^d}\left(K (s_1 f(t-x))+K(s_2 f(-x))\right)dx\right]\right|\\
\nonumber    &\leq \exp\left[-\int_{\R^d}\Re\{K (s_1 f(t-x)+s_2 f(-x))\}dx \right.\\
\nonumber    &\left.\wedge \int_{\R^d}\left(\Re\{K (s_1 f(t-x))\}+\Re\{K(s_2 f(-x))\}\right)dx\right]\\
\label{lmm3:eq2}    &\times \int_{\R^d}\left|K (s_1 f(t-x)+s_2 f(-x))-K (s_1 f(t-x))-K(s_2 f(-x))\right|dx.
\end{align}    
Recall that $\Re\{K\}\geq 0.$ Apply   Lemma \ref{lmm2}  to   
$$  \Re\{K (s_1 f(t-x)+s_2 f(-x))\}  \wedge  \Re\{K (s_1 f(t-x))+K(s_2 f(-x)) \}$$ to get
\begin{align*}
    &\int_{\R^d}\left[\Re\{K (s_1 f(t-x)+s_2 f(-x))\}\wedge 
\Re\{K (s_1 f(t-x))+K(s_2 f(-x))\}\right]dx\\
&\geq \int_{\R^d}\left(\sqrt{\Re\{K (s_1 f(t-x))\}}- \sqrt{\Re\{K(s_2 f(-x))\}}\right)^2 dx,
\end{align*} 
Then the right--hand side of \eqref{lmm3:eq2} is bounded from above by
\begin{align*} 
    & \exp\left[-\int_{\R^d}\left(\sqrt{\Re\{K (s_1 f(t-x))\}}- \sqrt{\Re\{K(s_2 f(-x))\}}\right)^2 dx\right]\\
    &\times \int_{\R^d}\left|K (s_1 f(t-x)+s_2 f(-x))-K (s_1 f(t-x))-K(s_2 f(-x))\right|dx\\
    & \leq   \exp\left[-\int_{\R^d}\left(\sqrt{\Re\{K (s_1 f(t-x))\}}- \sqrt{\Re\{K(s_2 f(-x))\}}\right)^2 dx\right]\\
    &\times 2  \int_{\R^d}\sqrt{\Re\{K (s_1 f(t-x))\}}\sqrt{\Re\{K(s_2 f(-x))\}}dx,
\end{align*}
where the last inequality follows from Lemma  \ref{lmm0}.

\end{proof}

Finally,  Theorem \ref{thm:main}  follows immediately from the lemma below:

\begin{lemma}
\label{lmm4} Let $\tilde{\rho}<1$ such that $|A_{\tilde{\rho}}|>0.$ Then for any $t\in A_{\tilde{\rho}}$ we have 
\begin{align}
\label{lmm4:eq1}
    \nonumber & \int_\R\int_\R  \left| \Cov(\mathbbm{1}\{X(t)>u\},\mathbbm{1}\{X(0)>v\})  \right| \mu(du)\mu(dv)\\
    &\leq \frac{2}{\pi^2} \left(\int_{\R_+}\frac{\sigma_f(s)}{s} e^{-(1-\tilde{\rho})\sigma_f^2(s)}ds\right)^2\tilde{\rho_t}.
\end{align}
\end{lemma}
\begin{proof}
By \cite[Theorem 2.1]{MaOeRappSpo21} we have
\begin{align*}
    &\int_\R\int_\R \Cov(\mathbbm{1}\{X(t)>u\},\mathbbm{1}\{X(0)>v\})\mu(du)\mu(dv)\\
    &= \frac{1}{4\pi^2}\int_{\R}\int_{\R} \frac{\vf_{(X(0),X(t))}(s_1,s_2)-\vf_{X(0)}(s_1)\vf_{X(0)}(s_2)}{-s_1 s_2}\hat{\mu}(s_1)\hat{\mu}(s_2)ds_1ds_2,
\end{align*}
where $\hat{\mu}(s)=\int_\R e^{isx} \mu(dx)$ is the characteristic function of a probability measure $\mu$ satisfying the inequality $|\hat{\mu}(s)| \le 1$ for all real $s$.
Applying Lemma \ref{lmm3} to the right--hand side of 
\begin{align*}
    &\int_\R\int_\R \left|  \Cov(\mathbbm{1}\{X(t)>u\},\mathbbm{1}\{X(0)>v\}) \right|    \mu(du)\mu(dv)\\
    &\le  \frac{1}{4\pi^2}\int_{\R}\int_{\R} \frac{ \left|  \vf_{(X(0),X(t))}(s_1,s_2)-\vf_{X(0)}(s_1)\vf_{X(0)}(s_2)\right|  }{|s_1 s_2| }  ds_1ds_2,
\end{align*}
 we get that  $$4 \pi^2 \int_\R\int_\R \left|  \Cov(\mathbbm{1}\{X(t)>u\},\mathbbm{1}\{X(0)>v\}) \right|    \mu(du)\mu(dv)$$ is bounded from above by
\begin{align*}
    &\int_{\R^2} \frac{\exp\left\{ -\int_{\R^d}\left(\sqrt{\Re\{K (s_1 f(t-x))\}}- \sqrt{\Re\{K(s_2 f(-x))\}}\right)^2 dx\right\}}{|s_1| |s_2|}\\
 &\times 2  \int_{\R^d}\sqrt{\Re\{K (s_1 f(t-x))\}}\sqrt{\Re\{K(s_2 f(-x))\}} \, dx ds_1 ds_2\\
 &=2\int_{\R^2}\exp\left\{-\sigma_f^2(s_1)-\sigma_f^2(s_2)+2\sigma_f(s_1)\sigma_f(s_2)\rho_t(s_1,s_2)\right\} \frac{\sigma_f(s_1)\sigma_f(s_2)\rho_t(s_1,s_2)}{|s_1| |s_2|} ds_1ds_2\\
 &\leq 2 \tilde{\rho_t}\int_{\R^2}\exp\left\{-\sigma_f^2(s_1)-\sigma_f^2(s_2)+2\sigma_f(s_1)\sigma_f(s_2)\tilde{\rho}\right\}\frac{\sigma_f(s_1)\sigma_f(s_2)}{|s_1| |s_2|} ds_1ds_2\\
 &= 2 \tilde{\rho_t}\int_{\R^2}\exp\left\{-(1-\tilde{\rho})\sigma_f^2(s_1)-(1-\tilde{\rho})\sigma_f^2(s_2)-(\sigma_f(s_1)-\sigma_f(s_2))^2\tilde{\rho}\right\} \frac{\sigma_f(s_1)\sigma_f(s_2)}{|s_1| |s_2|} ds_1ds_2\\
 &\leq 2 \tilde{\rho_t}\int_{\R^2}\exp\left\{-(1-\tilde{\rho})\sigma_f^2(s_1)-(1-\tilde{\rho})\sigma_f^2(s_2)\right\}\frac{\sigma_f(s_1)\sigma_f(s_2)}{|s_1| |s_2|} ds_1ds_2\\
 &=8 \left(\int_{\R_+}\frac{\sigma_f(s)}{s} e^{-(1-\tilde{\rho})\sigma_f^2(s)}ds\right)^2\tilde{\rho_t}.
\end{align*}
\end{proof}
The statement of  Theorem \ref{thm:main}  follows now easily by integrating the upper bound in Lemma \ref{lmm4} with respect to $t\in A_{\tilde{\rho}}$. Here, one writes
\begin{align*} 
& \int\limits_{\R^d}  \int\limits_\R\int\limits_\R \left| \Cov(\mathbbm{1}\{X(t)>u\},\mathbbm{1}\{X(0)>v\})  \right| \mu(du)\mu(dv) \, dt\\
& =\left( \int\limits_{A_{\tilde{\rho}}}   + \int\limits_{A^c_{\tilde{\rho}}}  \right)  \int\limits_\R\int\limits_\R \left| \Cov(\mathbbm{1}\{X(t)>u\},\mathbbm{1}\{X(0)>v\})  \right| \mu(du)\mu(dv) \, dt, \end{align*}
 where the  integral  over $A_{\tilde{\rho}}$ is finite by Lemma \ref{lmm4}, and 
 the  integral  over $A^c_{\tilde{\rho}}$ is finite due to the finiteness of the volume of  $A^c_{\tilde{\rho}}$, since the integrand  $$ \int\limits_\R\int\limits_\R \left| \Cov(\mathbbm{1}\{X(t)>u\},\mathbbm{1}\{X(0)>v\})  \right|   \mu(du)\mu(dv) \le 1.$$

%\begin{align*}
%    &\frac{1}{|s_1s_2|}\left|\E e^{i \left(s_1X(t)+s_2X(0)\right)}-\E e^{i s_1 X(t)} \E e^{i s_2X(0)}\right|\\
%    &\leq 2\exp\left(-|s_1|^{\alpha} \sigma^2_f-|s_2|^{\alpha} \sigma^2_f +2 |s_1|^{\alpha/2}|s_2|^{\alpha/2} \rho_f(f)  \sigma^2_f\right) |s_1|^{\alpha/2-1}|s_2|^{\alpha/2-1}\rho_f(t)\sigma^2_f
%\end{align*}
%and
%\begin{align*}
%    &\int_{\R^2}\frac{1}{|s_1s_2|}\left|\E e^{i \left(s_1X(t)+s_2X(0)\right)}-\E e^{i s_1 X(t)} \E e^{i s_2X(0)}\right| ds_1 ds_2\\
%    &\leq 2\rho_f(t)\sigma^2_f \int_{\R^2}\exp\left(-|s_1|^{\alpha} \sigma^2_f-|s_2|^{\alpha} \sigma^2_f +2 |s_1|^{\alpha/2}|s_2|^{\alpha/2} \rho_f(t)  \sigma^2_f\right) |s_1|^{\alpha/2-1}|s_2|^{\alpha/2-1} ds_1 ds_2\\
%    &\leq  \frac{32 \pi}{\alpha^2} \frac{\rho_f(t)}{\sqrt{1-\rho^2_f(t)}}
%\end{align*}
%-----------------------------------------------------------------------------------------------------------------------------------------------------------------------------------------------------------------------

\section*{Acknowledgement}
This research was supported by the Grant No. 39087941 of the German Research Society.

%-----------------------------------------------------------------------------------------------------------------------------------------------------------------------------------------------------------------------
%

%\bibliographystyle{abbrv}
%\bibliography{abbrev,geostoch}

\newcommand{\noopsort}[1]{} \newcommand{\printfirst}[2]{#1}
  \newcommand{\singleletter}[1]{#1} \newcommand{\switchargs}[2]{#2#1}

\end{document}